\documentclass[11pt]{amsart}

\usepackage[utf8]{inputenc}
\usepackage{lmodern}
\usepackage[T1]{fontenc}

\usepackage{mathtools}
\mathtoolsset{showonlyrefs}
\usepackage{amssymb}
\usepackage{amsthm}
\usepackage{graphicx}
\usepackage{tikz}

\usepackage[margin=1.2in]{geometry}
\usepackage{hyperref}

\def\pr{\mathbb{P}}
\def\Pr{\pr}
\def\P{\pr}

\def\E{\mathbb{E}}

\def\ind{\mathbb{I}}

\def\E{\mathbb{E}}

\def\P{\mathbb{P}}

\def\eps{\varepsilon}
\def\del{\delta}

\def\1{\mathbf{1}}

\def\lam {\lambda}

\def\tce{t_c + \eps}
\def\tce2{t_c + \frac{\eps}{2}}

\def\erdos{Erd\H{o}s }

\newtheorem*{theorem*}{Theorem}
\newtheorem{theorem}{Theorem}
\newtheorem{lemma}[theorem]{Lemma}
\newtheorem{cor}[theorem]{Corollary}

\newtheorem*{defn*}{Definition}
\newtheorem{prop}[theorem]{Proposition}
\newtheorem*{prop*}{Proposition}
\newtheorem{conj}[theorem]{Conjecture}
\newtheorem*{conj*}{Conjecture}
\newtheorem{claim}[theorem]{Claim}

\newtheorem*{fact*}{Fact}

\def\endofClaim{\hfill\scalebox{.6}{$\Box$}}

\newcommand{\oldqed}{}

\newenvironment{claimproof}[1][Proof]{
  \renewcommand{\oldqed}{\qedsymbol}
  \renewcommand{\qedsymbol}{\endofClaim}
  \begin{proof}[#1]
}{
  \end{proof}
  \renewcommand{\qedsymbol}{\oldqed}
}

\def\ind{^{\mathrm{ind}}}
\def\match{^{\mathrm{match}}}
\def\Zind{Z^{\mathrm{ind}}}
\def\Zmatch{Z^{\mathrm{match}}}

\def\alpind{\alpha^{\mathrm{ind}}}
\def\alpmatch{\alpha^{\mathrm{match}}}

\def\CLdn{\mathit{CL}_{d,n}}
\def\mperf{m_{\mathrm{perf}}}
\def\heawood{\mathit{HW}}
\def\heawoodn{\mathit{HW_n}}
\DeclarePairedDelimiter{\abs}{\lvert}{\rvert}

\begin{document}

\title{Tight bounds on the coefficients of partition functions via stability}
\author{Ewan Davies}
\author{Matthew Jenssen}
\author{Will Perkins}
\author{Barnaby Roberts}
\address{London School of Economics}
\email{\{e.s.davies,m.o.jenssen,b.j.roberts\}@lse.ac.uk }
\address{University of Birmingham}
\email{math@willperkins.org}
\thanks{WP supported in part by EPSRC grant EP/P009913/1.}
\date{\today}

\begin{abstract}
Partition functions arise in statistical physics and probability theory as the normalizing constant of Gibbs measures and in combinatorics and graph theory as graph polynomials. For instance the partition functions of the hard-core model and monomer-dimer model are the independence and matching polynomials respectively.     

We show how stability results follow naturally from the recently developed occupancy method for maximizing and minimizing physical observables over classes of regular graphs, and then show these stability results can be used to obtain tight extremal bounds on the individual coefficients of the corresponding partition functions.  

As applications, we prove new bounds on the number of independent sets and matchings of a given size in regular graphs.  For large enough graphs and almost all sizes, the bounds are tight and confirm the Upper Matching Conjecture of Friedland, Krop, and Markstr{\"o}m and a conjecture of Kahn on independent sets for a wide range of parameters. Additionally we prove tight bounds on the number of $q$-colorings of cubic graphs with a given number of monochromatic edges, and tight bounds on the number of independent sets of a given size in cubic graphs of girth at least $5$. 
\end{abstract}

\maketitle

\section{Introduction}\label{sec:intro}

The matching polynomial (or matching generating function) of a graph $G$ is the function
\[ \Zmatch_G(\lam) = \sum_{M \in \mathcal M(G)} \lam^{|M|}\,, \]
where the sum is over $\mathcal M(G)$, the set of all matchings in the graph $G$. Analogously, the independence polynomial of a graph $G$ is 
\[ \Zind_G(\lam) = \sum_{I \in \mathcal I(G)} \lam^{|I|}\,, \]
where $\mathcal I(G)$ is the set of all independent sets of $G$.  In statistical physics, $\Zmatch_G(\lam)$ and $\Zind_G(\lam)$ are the partition functions of the monomer-dimer and hard-core models respectively.  

 The following theorems give a tight upper bound on $\Zmatch_G(\lam)$ and $\Zind_G(\lam)$ over the family of all $d$-regular graphs.

\begin{theorem}[Davies, Jenssen, Perkins, Roberts \cite{Davies2015}]
\label{thm:matchPart}
For any $d$-regular graph $G$ and any $\lam>0$,
\[ \frac{1}{|V(G)|} \log \Zmatch_G(\lam) \le  \frac{1}{2d} \log \Zmatch_{K_{d,d}}(\lam)  \,.  \]
\end{theorem}

\begin{theorem}[Kahn~\cite{kahn2001entropy,kahn2002entropy}, Galvin--Tetali~\cite{galvin2004weighted}, Zhao \cite{zhao2010number}]
\label{thm:indPart}
For any $d$-regular graph $G$ and any $\lam>0$,
\[ \frac{1}{|V(G)|} \log \Zind_G(\lam) \le  \frac{1}{2d} \log \Zind_{K_{d,d}}(\lam)  \,.  \]
\end{theorem}

In particular, if we set $\lam =1$, then the two theorems say that if $2d$ divides $n$, then the total number of matchings and the number of independent sets in any $d$-regular graph on $n$ vertices is at most that of the graph $H_{d,n}$ consisting of $n/2d$ copies of $K_{d,d}$.   For much more on such extremal problems for regular graphs, see the notes of Galvin~\cite{galvin2014three}, survey of Zhao~\cite{yufeiSurvey}, and paper of Csikv{\'a}ri~\cite{csikvari2016extremal}. 

In this paper we will address two strengthenings of results of the form above and how they are related.  The first possible strengthening of Theorems~\ref{thm:matchPart} and~\ref{thm:indPart} is that $H_{d,n}$ might maximize the polynomials $\Zmatch$ and $\Zind$ on the level of each individual coefficient.  Let $m_k(G)$ and $i_k(G)$ denote the number of matchings and independent sets of size $k$ respectively in a graph $G$. Then we can write $\Zmatch_G(\lam) = \sum_{k \ge 0} m_k(G) \lam^k$.   Kahn and Friedland, Krop, and Markstr{\"o}m conjectured the following.

\begin{conj}[Kahn \cite{kahn2001entropy}]
\label{conj:IndGivenSize}
Let $2d$ divide $n$. Then for any $d$-regular $G$ on $n$ vertices, and any $1\le k \le n/2$, 
\[ i_{k}(G) \le i_k(H_{d,n}) \,. \]
\end{conj}

\begin{conj}[Friedland, Krop, Markstr{\"o}m \cite{friedland2008number}]
\label{conj:UMC}
Let $2d$ divide $n$. Then for any $d$-regular $G$ on $n$ vertices, and any $1\le k \le n/2$, 
\[ m_{k}(G) \le m_k(H_{d,n}) \,. \]
\end{conj}

These conjectures were known to be true in a very small number of cases ($k \le 4$ and $k=n/2$).  Several approximate versions of the two conjectures have been proved by Carroll, Galvin, and Tetali~\cite{carroll2009matchings}, Ilinca and Kahn~\cite{ilinca2013asymptotics}, Perkins~\cite{perkins2016birthday}, and Davies, Jenssen, Perkins and Roberts~\cite{Davies2015}.  The first three bounds were in general off by a factor exponential in $n$, while the fourth bound was off by a factor $2 \sqrt{n}$.  The case $k=n/2$ for matchings follows from Bregman's theorem~\cite{bregman1973some} on the permanents of $0/1$ matrices with given row sums. (The case of independent sets in non-regular graphs with given minimum degree was proved by Gan, Loh, and Sudakov~\cite{gan2015maximizing}).   

Our first main result is that the conjectures of Kahn and Friedland, Krop, and Markstr{\"o}m hold for a wide range of parameters.
\begin{theorem}\label{thm:exactslices}
For all $d \geq 2$ and $ \eps >0$ there exists $N=N(d, \eps)$ such that the following holds.
Suppose that $n \geq N$ and $n$ is divisible by $2d$. Let $G$ be any $d$-regular graph on $n$ vertices. Then for all $k>\eps n$,
\begin{align}
i_k(G) &\leq i_k(H_{d,n})\,,\\
m_k(G) &\leq m_k(H_{d,n})\,.
\end{align}
\end{theorem}
 In other words, for sufficiently large graphs, and almost all values of $k$, Conjectures~\ref{conj:IndGivenSize} and~\ref{conj:UMC} hold.

The second strengthening of Theorems~\ref{thm:matchPart} and~\ref{thm:indPart} we consider is the question of uniqueness and stability.  The proofs of those theorems imply that equality is attained only by $K_{d,d}$ and unions of copies of $K_{d,d}$.  The combinatorial notion of \textit{stability} goes beyond this and states that any graph with a near extremal matching or independence polynomial must be `close' to one of the extremal graphs under some natural distance on graphs.  In extremal combinatorics, the theory of stability for Tur\'an-type problems was developed by \erdos and Simonovits \cite{erdos1966some,simonovits1968method}.  Stability has also proved very useful in the extremal theory of dense graphs and graph limits~\cite{pikhurko2010analytic}, and in other areas of combinatorics (e.g.~\cite{friedgut2008measure,keevash2004stability,mubayi2007structure}). 

Let $\delta_{\circ}(G,H)$ denote the sampling distance between two bounded degree graphs (see Section~\ref{sec:stable} below and Lov{\'a}sz's book on graph limits~\cite{lovasz2012large}). This is a distance function that metrizes Benjamini--Schramm convergence~\cite{benjamini2001recurrence}.  Our next result is that Theorems~\ref{thm:matchPart} and~\ref{thm:indPart} are stable under the $\delta_{\circ}(\cdot,\cdot)$ distance.
\begin{theorem}\label{thm:stability}
For any $d\ge 2$ there exist continuous functions $s\match(d, \lam)$, $s\ind(d ,\lam)$ which satisfy $s\match(d,0)=s\ind(d,0)=0$, and which are strictly increasing in $\lam$ such that the following holds.
For any $d$-regular graph $G$, 
\begin{align*}
\frac{1}{|V(G)|} \log \Zmatch_G(\lambda) &\leq \frac{1}{2d} \log\Zmatch_{K_{d,d}}(\lam) - s\match(d,\lam) \cdot\delta_{\circ}(G,K_{d,d})\,, \\
\frac{1}{|V(G)|} \log \Zind_G(\lambda) &\leq  \frac{1}{2d} \log \Zind_{K_{d,d}}(\lam) - s\ind(d,\lam) \cdot\delta_{\circ}(G,K_{d,d}) \,.
\end{align*}
\end{theorem}
 Up to constant factors $\delta_{\circ}(G,K_{d,d})$ is simply the fraction of vertices of $G$ that are not in a $K_{d,d}$ component, and so the reader may think of $\delta_{\circ}(G,K_{d,d})$ in this way, but in the general setting of extremal problems for bounded degree graphs, using a distance compatible with Benjamini--Schramm convergence is very natural (see e.g. \cite{csikvari2014lower,lelarge2017counting}) and so we stick with this notation. 

A stability version of Theorem~~\ref{thm:indPart} for independent sets at $\lam=1$ was recently proved by Dmitriev and Dainiak~\cite{dmitriev2016uniqueness} by examining the details of the entropy method proof.  Our approach is more probabilistic and proceeds via the occupancy method of~\cite{Davies2015} and properties of linear programming. It also yields monotonicity of the functions $s\match(d,\lam), s\ind(d,\lam)$ in $\lam$ which is helpful in proving the results below and does not obviously follow from entropy methods.

The method we use to prove extremality and stability via probabilistic observables and linear programming, and the method we use to deduce exact bounds on individual coefficients from stability versions of extremal results are both generally applicable.  We provide two further examples to illustrate this.

Let $\heawood$ be the Heawood graph, the unique $(3,6)$-cage graph (see Figure~\ref{fig:Heawood}).  For $n$ divisible by $14$, let $\heawoodn$ be the graph consisting of a union of $n/14$ copies of $\heawood$. Perarnau and Perkins~\cite{perarnau2016counting} proved that $\frac{1}{|V(G)|} \log \Zind_G(\lam)$ is maximized over all cubic graphs of girth at least $5$ by $\heawood$.  Here we use this and a corresponding stability result to prove tight upper bounds on the number of independent sets of a given size in such graphs.
\begin{theorem}\label{thm:girthslices}
For all $\eps >0$ there exists $N(\eps)$ so that the following is true. Suppose that $n \geq N$ and $n$ is divisible by  $14$. Let $G$ be any $3$-regular graph of girth at least $5$ on $n$ vertices. Then for all $k > \eps n$,
\begin{align}
i_k(G) &\leq i_k(\heawoodn)\,.
\end{align}
\end{theorem}

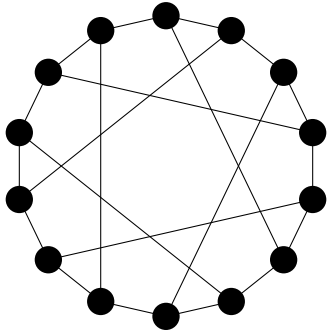
\begin{figure}
\begin{tikzpicture}
    \node[shape=circle,fill=black] (A1) at (0,2) {};
    \node[shape=circle,fill=black] (A2) at (0,-2) {};
   \node[shape=circle,fill=black] (C1) at (0.868,1.802) {};
    \node[shape=circle,fill=black] (C2) at (0.868,-1.802) {};
   \node[shape=circle,fill=black] (C3) at (-0.868,1.802) {};
    \node[shape=circle,fill=black] (C4) at (-0.868,-1.802) {};
    
    \node[shape=circle,fill=black] (D1) at (1.564,1.247) {};
    \node[shape=circle,fill=black] (D2) at (1.564,-1.247) {};
   \node[shape=circle,fill=black] (D3) at (-1.564,1.247) {};
    \node[shape=circle,fill=black] (D4) at (-1.564,-1.247) {};
    \node[shape=circle,fill=black] (E1) at (1.95,0.445) {};
    \node[shape=circle,fill=black] (E2) at (1.95,-0.445) {};
    \node[shape=circle,fill=black] (E3) at (-1.95,0.445) {};
    \node[shape=circle,fill=black] (E4) at (-1.95,-0.445) {};

\draw plot [smooth, tension=2] coordinates { (A1) (C1) };
\draw plot [smooth, tension=2] coordinates { (C1) (D1) };
\draw plot [smooth, tension=2] coordinates { (D1) (E1) };
\draw plot [smooth, tension=2] coordinates { (E1) (E2) };
\draw plot [smooth, tension=2] coordinates { (E2) (D2) };
\draw plot [smooth, tension=2] coordinates { (D2) (C2) };
\draw plot [smooth, tension=2] coordinates { (C2) (A2) };
\draw plot [smooth, tension=2] coordinates { (A2) (C4) };
\draw plot [smooth, tension=2] coordinates { (C4) (D4) };
\draw plot [smooth, tension=2] coordinates { (D4) (E4) };
\draw plot [smooth, tension=2] coordinates { (E4) (E3) };
\draw plot [smooth, tension=2] coordinates { (E3) (D3) };
\draw plot [smooth, tension=2] coordinates { (D3) (C3) };
\draw plot [smooth, tension=2] coordinates { (C3) (A1) };

\draw plot [smooth, tension=2] coordinates { (E4) (C1) };
\draw plot [smooth, tension=2] coordinates { (E3) (C2) };
\draw plot [smooth, tension=2] coordinates { (E1) (D3) };
\draw plot [smooth, tension=2] coordinates { (E2) (D4) };
\draw plot [smooth, tension=2] coordinates { (C3) (C4) };
\draw plot [smooth, tension=2] coordinates { (A1) (D2) };
\draw plot [smooth, tension=2] coordinates { (A2) (D1) };
\end{tikzpicture}
\caption{The Heawood graph, $\heawood$, the smallest 3-regular graph with girth 6.}
\label{fig:Heawood}
\end{figure}

Next we consider colorings of regular graphs.  Galvin and Tetali~\cite{galvin2004weighted} conjectured (and proved for bipartite graphs) that $H_{d,n}$ maximizes the number of proper $q$-colorings of any $d$-regular graph on $n$ vertices.  The general conjecture is still open, but the case $d=3$ and arbitrary $q$ was proved in~\cite{davies2016extremes}.  To view this conjecture and these results in the setting of partition functions and graph polynomials, we consider the $q$-color Potts model from statistical physics.    Let $c^q_k(G)$ be the number of $q$-colorings of the vertices of $G$ with exactly $k$ monochromatic edges.  Then the function
\[ Z^q_G(\lam) = \sum_{k \ge 0} c^q_k(G) \lam^k   \]
is the Potts model partition function (the standard formulation has $\lam$ replaced by $e^{-\beta}$). The conjecture on proper colorings is that $H_{d,n}$ maximizes the coefficient $c_0^q(G)$ over $n$-vertex, $d$-regular graphs.  Here we show that for cubic graphs, $H_{3,n}$ maximizes $c_k^q(G)$ for a wide range of $k$.
\begin{theorem}\label{thm:colorslices}
For all $q \ge 3$, $\eps>0$, there exists $N(q, \eps)$ so that the following is true.  Suppose that $n \geq N$ and $n$ is divisible by $6$. Let $G$ be any $3$-regular graph on $n$ vertices. Then for $ 0 \le k \le (1-\eps) 3n/2q $ we have
\begin{align*}
c^q_k(G) &\le c^q_k(H_{3,n})\,.
\end{align*}
\end{theorem}

The range of $k$ is essentially optimal by the following observation. 
Note that the expected number of monochromatic edges in a uniformly random assignment of $q$ colors to the vertices of a $d$-regular graph is $nd/2q$.  For an $n$-vertex graph $G$ the sum of all the coefficients $c_k$ is equal to $q^n$, and so no single graph can maximize all of the coefficients. 
We expect that a union of cliques maximizes most coefficients $c_k$ where $k \ge nd/2q$.  

Moreover, while a result similar to Theorem~\ref{thm:colorslices} is undoubtedly true for the case $q=2$, the exact statement is not true ($H_{3,n}$ has no $2$-colorings with exactly $1$ or $2$ monochromatic edges) and so instead of dealing with unenlightening technicalities we omit this case.

\paragraph{\bf Outline of the remainder of the paper}

\begin{itemize}
\item In Section~\ref{sec:probability} we define the related probabilistic models from statistical physics and present some of their probabilistic properties.
\item In Section~\ref{sec:stable} we define our notion of stability for partition functions of bounded degree graphs and prove Theorem~\ref{thm:stability} for matchings.
\item In Section~\ref{sec:transfer} we prove our first transference result, showing how stability theorems imply exact results for coefficients.  Here we prove Theorem~\ref{thm:exactslices} for matchings.
\item In Section~\ref{sec:hierarchy} we describe a hierarchy of different types of extremal results for partition functions and graph polynomials and present some open problems.
\item In Section~\ref{sec:appendix} we give the modifications of the proofs in Sections~\ref{sec:stable} and~\ref{sec:transfer} necessary to prove Theorems~\ref{thm:exactslices},~\ref{thm:stability}, and~\ref{thm:girthslices} for independent sets and Theorem~\ref{thm:colorslices} for colorings. 
\end{itemize}

\section{Probabilistic properties of partition functions}
\label{sec:probability}

\subsection{Occupancy fractions and observables}
Each of the partition functions defined above comes from a model in statistical physics.  The hard-core model on a graph $G$ at fugacity $\lam$ is a random independent set $\mathbf I$ chosen from $G$  with probability
\begin{align*}
\pr_{G,\lam}(\mathbf I =I) &= \frac{\lam^{|I|}}{\Zind_G(\lam)} \,.
\end{align*}
(We indicate the graph and value of the fugacity in the subscript of probabilities and expectations but drop it from the notation when they are clear from context). The partition function $\Zind_G(\lam)$ is the normalizing constant of the probability distribution -- it ensures that the sum of $\pr_{G,\lam}(\mathbf I=I)$ over all independent sets is $1$.  In statistical physics terminology, the scaling $\frac{1}{|V(G)|} \log \Zind_G(\lam)$ is the free energy (or free energy density) and is a convenient scaling to allow us to compare partition functions on graphs with different numbers of vertices.  In particular for any graph $G$ the free energy density of $G$ is the same as the free energy density of multiple copies of $G$.

Likewise the monomer-dimer model at fugacity $\lam$ is a random matching $\mathbf M$ chosen with probability 
\begin{align*}
\pr_{G,\lam}(\mathbf M = M) &= \frac{\lam^{|M|}}{\Zmatch_G(\lam)}\,,
\end{align*}
and the Potts model is a random vertex-coloring $\boldsymbol \chi$ (not necessarily proper) chosen with probability
\begin{align*}
\pr_{G,\lam}(\boldsymbol \chi = \chi) &= \frac{\lam^{m(G,\chi)}}{Z^q_G(\lam)}\,,
\end{align*}
where $m(G,\chi)$ is the number of monochromatic edges in $G$ with coloring $\chi$.  If  $\lam >1$, we say the model is ferromagnetic (monochromatic edges preferred); if $\lam \in [0,1)$ we say the model is anti-ferromagnetic (bichromatic edges preferred); and in the specific case $\lam =0$ (and $G$ has at least one proper $q$-coloring), then $\boldsymbol \chi$ is a uniformly chosen proper $q$-coloring of $G$. 

Associated to these models are  physical \textit{observables}: expectations of real-valued functions of the random configurations.  For instance, the occupancy fraction of the hard-core model is the expected fraction of vertices of the graph that appear in the random independent set:
\begin{align*}
\alpind_G(\lam) &= \frac{1}{|V(G)|} \E_{G,\lam} | \mathbf I| \,.
\end{align*}
Likewise, the occupancy fraction of the monomer-dimer model is the expected fraction of edges that appear in the random matching:
\begin{align*}
\alpmatch_G(\lam) &= \frac{1}{|E(G)|} \E_{G,\lam} | \mathbf M|\,,
\end{align*}
and the internal energy of the Potts model is the expected fraction of monochromatic edges:
\begin{align*}
\beta^q_G(\lam) &=  \frac{1}{|E(G)|} \E_{G,\lam} |m(G, \boldsymbol \chi)|\,.
\end{align*}

For our purposes, the importance of these particular observables is the fact that they can be expressed in terms of derivatives of the corresponding log partition functions:
\begin{align*}
\alpind_G(\lam) = \frac{1}{|V(G)|} \E_{G,\lam}|\mathbf I|  &= \frac{1}{|V(G)|} \frac{\sum _{I \in \mathcal I(G)} |I| \lam ^{|I|}   }{\Zind_G(\lam)  }  \\
&= \frac{1}{|V(G)|} \frac{\lam}{\Zind_G(\lam)} \frac{\partial }{\partial \lam}\Zind_G(\lam)  \\
&= \frac{\lam}{|V(G)|}\frac{\partial}{\partial \lam}\log\Zind_G(\lam) \,,
\end{align*}
and similarly,
\begin{align*}
\alpmatch_G(\lam) &= \frac{\lam}{|E(G)|}\frac{\partial}{\partial \lam}\log\Zmatch_G(\lam)\,,  \\
\beta^q_G(\lam) &=  \frac{\lam}{|E(G)|}\frac{\partial }{\partial \lam}\log Z^q_G(\lam) \, .
\end{align*}

The method of~\cite{Davies2015} gives the following, from which we derive Theorem~\ref{thm:matchPart} and an alternative proof of Theorem~\ref{thm:indPart}.

\begin{theorem}[Davies, Jenssen, Perkins, Roberts \cite{Davies2015}]
\label{thm:indOcc}
For any $d$-regular graph $G$, and any $\lam>0$,
\begin{align*}
\alpind_G(\lam) &\le \alpind_{K_{d,d}}(\lam) \,, 
\shortintertext{and}
\alpmatch_G(\lam) &\le \alpmatch_{K_{d,d}}(\lam)\,. 
\end{align*}
\end{theorem}
The above result is a strengthening of Theorems~\ref{thm:matchPart} and~\ref{thm:indPart} for the partition function with the additional benefit of a natural probabilistic interpretation.  The analogous result for independent sets in cubic graphs of girth at least $5$ was proved in~\cite{perarnau2016counting}. 
We also know that in the anti-ferromagnetic Potts model on cubic graphs, that $K_{3,3}$ minimizes the internal energy, which in turn implies that $K_{3,3}$ maximizes the free energy.
\begin{theorem}[Davies, Jenssen, Perkins, Roberts \cite{davies2016extremes}]
\label{thm:colorOcc}
For any $3$-regular graph $G$, any $q\ge 2$, and any $\lam\in [0,1]$, 
\begin{align*}
\beta^q_G(\lam) & \ge \beta^q_{K_{3,3}}(\lam)\,,
\end{align*}
and so for every $\lam \in [0,1]$,
\begin{align*}
\frac{1}{|V(G)|} \log Z_G^q(\lam) &\le \frac{1}{6} \log Z^q_{K_{3,3}}(\lam)\,. 
\end{align*}
\end{theorem}

\subsection{The canonical ensemble}
The random configurations $\mathbf I$, $\mathbf M$, and $\boldsymbol \chi$ defined above come from the \textit{grand canonical ensemble} in the language of statistical physics, and these models have several appealing mathematical properties, most notably the spatial Markov property.  
The individual coefficients of $\Zind_G(\lam)$, $\Zmatch_G(\lam)$, $Z_G^q(\lam)$, however, are associated directly to a different probabilistic model,  the \textit{canonical ensemble}.  In this setting $\mathbf I_k$ is a uniformly random independent set of size $k$ drawn from $G$; $\mathbf M_k$ is a uniformly random matching of size $k$; and $\boldsymbol \chi_k$ is a uniformly random coloring of $V(G)$ with $q$ colors that has exactly $k$ monochromatic edges.  The partition function of the canonical ensemble is simply $i_k(G)$, the number of independent sets of size $k$, or likewise $m_k(G)$ or $c_k^q(G)$. 

The two ensembles are closely related; if we condition on the event that $| \mathbf I|=k$ in the grand canonical ensemble then we obtain the distribution $\mathbf I_k$ (assuming of course that $G$ has at least one independent set of size $k$).  Moreover if we choose $\lam$ so that $\E_{\lam} | \mathbf I| = k$, then we would expect $\mathbf I$ and $\mathbf I_k$ to behave quite similarly.  In physics the two ensembles are considered to be essentially equivalent for most purposes; one way to understand this is by imagining running the canonical ensemble on a large lattice or otherwise periodic graph; if we then zoom in to a small block of the large graph, the induced distribution looks almost identical to the distribution of the grand canonical ensemble on the small block.  We use a precise instantiation of this idea in our proofs in Section~\ref{sec:transfer}.

Theorems~\ref{thm:matchPart} and~\ref{thm:indPart} can be viewed as tight upper bounds on partition functions from the grand canonical ensemble, whereas Conjectures~\ref{conj:IndGivenSize} and~\ref{conj:UMC} (and Theorem~\ref{thm:exactslices}) discuss tight upper bounds on the `partition function' of the canonical ensemble. 
The method we use to prove Theorem~\ref{thm:exactslices} builds upon the occupancy method of~\cite{Davies2015}, which relies heavily on the spatial Markov property of the grand canonical ensemble. 
The canonical ensemble does not have this property, and so these techniques do not apply directly. 
As the authors showed in~\cite{Davies2015} with a proof of an asymptotic version of Conjecture~\ref{conj:UMC}, the similarities between these ensembles mean that bounds relating to the grand canonical ensemble can give approximate versions of the results of this paper.
Here we go further, exploiting stability to obtain tight results.
In Section~\ref{sec:hierarchy} we consider general relations between extremal bounds for the two ensembles.

\section{Stability}
\label{sec:stable}

\subsection{Sampling distance}

There are many notions of distances between graphs.  We use the notion of \textit{sampling distance} which is closely related to the notion of Benjamini--Schramm convergence for sequences of bounded-degree graphs, see~\cite{benjamini2001recurrence,lovasz2012large}.
Let $G$ be a finite graph of maximum degree at most $d$.  Let $\rho_r(G)$ be the probability distribution on rooted graphs of depth at most $r$ induced by choosing a vertex $v$ uniformly at random from $V(G)$ and taking the depth $r$ ball around $v$. Then we can define the \textit{sampling distance} of two graphs $G$, $G'$ to be 
\begin{align*}
\delta_{\circ}(G, G^\prime) &= \sum_{r=1}^\infty 2^{-r} \| \rho_r(G) - \rho_r(G^\prime) \|_{TV}\,,
\end{align*}
where $\| \cdot \|_{TV}$ is the total variation distance of two probability distributions on the same ground set.  A sequence of graphs $G_n$ is Benjamini--Schramm convergent if it is a Cauchy sequence in the distance function $\delta_{\circ}(\cdot, \cdot)$. 

A simple fact about $\delta_{\circ}(\cdot, \cdot)$ is that if $H$ is any graph and $H^\prime $ consists of $k \ge 1$ copies of $H$, then $\delta_{\circ} (H, H^\prime) =0$ (in particular, $\delta_{\circ}(\cdot, \cdot)$ is not a metric).

\subsection{Stability results}\label{sec:deducestability}

Theorem~\ref{thm:stability} is a direct consequence  of the methods of~\cite{Davies2015} which show that $K_{d,d}$ is the unique connected graph which maximizes the occupancy fraction.
These methods have been used to give tight bounds on quantities analogous to the occupancy fraction in a variety of models and classes of graphs~\cite{cohen2016widom,cutler2016minimizing,davies2016extremes,davies2016average,perarnau2016counting}, reducing extremal graph theory problems to solving linear programs over the space of probability distributions on local configurations (for more detail see the aforementioned references and Section~\ref{sec:proofs-stability} below). 
In cases in which there is a unique extremal distribution arising from a unique connected graph,  a stability statement analogous to Theorem~\ref{thm:stability} follows from complementary slackness and basic calculus. 
If desired, an explicit form for the constants (e.g. $c\ind(d, \lam),\, c\match(d, \lam)$ below) can be computed via this method, but we omit such calculations. 

The statement for the occupancy fraction from which we begin exists implicitly in~\cite{Davies2015}, but here we give more detail. 

\begin{prop}\label{prop:occ-stability}
There exists continuous functions $c\ind(d, \lam),\, c\match(d, \lam)\ge 0$ which are strictly positive when $\lam>0$ such that for any $d$-regular graph $G$,
\begin{gather}
 \alpmatch_G(\lam) \le \alpmatch_{K_{d,d}}(\lam) - c\match(d, \lam) \cdot\delta_{\circ}(G,K_{d,d})\,,\\
 \alpind_G(\lam) \le \alpind_{K_{d,d}}(\lam) - c\ind(d, \lam) \cdot\delta_{\circ}(G,K_{d,d})\,.
\end{gather}
\end{prop}

We prove Proposition~\ref{prop:occ-stability} for matchings in Section~\ref{sec:proofs-stability} and give the give the required adaptation to prove the statement for independent sets in Section~\ref{sec:appendix}. 
We give the short derivation of Theorem~\ref{thm:stability} from Proposition~\ref{prop:occ-stability} here.

\begin{proof}[Proof of Theorem~\ref{thm:stability}]
Using Proposition~\ref{prop:occ-stability} and simple calculus, we have
\begin{align*}
\frac{1}{|V(G)|} \log \Zmatch_G(\lam) &=  \frac{1}{|V(G)|}\log \Zmatch_G(0) +  \frac{1}{|V(G)|}\int_{0}^\lam \big(\log \Zmatch_G(t)\big)' \, dt \\
 &=   \frac{d}{2}\cdot\int_{0}^\lam \frac{\alpmatch_G(t)}{t} \, dt \\
 &\le \frac{d}{2}\cdot\int_{0}^\lam \left(\frac{\alpmatch_{K_{d,d}}(t)}{t} - \frac{c\match(d,t) \cdot\delta_{\circ}(G,K_{d,d})}{t}\right)\, dt \\
 &= \frac{1}{2d}\log \Zmatch_{K_{d,d}}(\lam) - \delta_{\circ}(G,K_{d,d})\cdot \frac{d}{2}\int_{0}^\lam \frac{c\match(d,t)}{t}\, dt\,.
\end{align*}
Hence we may take $s\match(d,\lam) = \frac{d}{2}\int_{0}^\lam \frac{c\match(d,t)}{t}\, dt$, which proves the theorem since $c\match(d,t)>0$ for all $t>0$. 
The proof for independent sets is essentially the same.
\end{proof}

As a corollary of Theorem~\ref{thm:stability}, we obtain the following stability version of Bregman's theorem on perfect matchings applied to regular graphs.
\begin{cor}
\label{cor:bregmanStable}
There exists a constant $c(d)$ so that for any $d$-regular $G$,
\[ \frac{1}{|V(G)|} \log  \mperf(G) \le \frac{1}{2d}  \log \mperf(K_{d,d}) - c(d)\cdot\delta_{\circ}(G,K_{d,d})\,.\]
\end{cor}
(An analogue holds for independent sets with an elementary proof). 
\begin{proof}
We can write $\mperf(G)$ in terms of the monomer-dimer partition function:
\begin{align*}
\frac{1}{|V(G)|} \log  \mperf(G) &= \lim_{\lam \to \infty} \frac{1}{|V(G)|} \log \Zmatch_G(\lam) - \frac{\log \lam}{2}\,.
\end{align*}
(If $G$ has no perfect matching then the right-hand side tends to $- \infty$). Now applying Theorem~\ref{thm:stability},
\begin{align*}
\frac{1}{|V(G)|} \log  \mperf(G) &= \lim_{\lam \to \infty} \frac{1}{|V(G)|} \log \Zmatch_G(\lam) - \frac{\log \lam}{2} \\
&\le \lim_{\lam \to \infty} \frac{1}{2d} \log \Zmatch_{K_{d,d}}(\lam) - \frac{\log \lam}{2} -  s\match(d,\lam) \cdot\delta_{\circ}(G,K_{d,d})\\
&= \frac{1}{2d}  \log \mperf(K_{d,d}) - c(d)\cdot\delta_{\circ}(G,K_{d,d})\,,
\end{align*}
where we take $c(d) = \lim_{\lam \to \infty} s\match(d,\lam)$. The limit exists because $s\match(d,\lam)$ is strictly increasing in $\lam$ and is bounded (for if were not bounded, the above inequality for $\lam$ large would contradict the existence of a perfect matching in a graph $G$ with $\delta_{\circ}(G,K_{d,d})>0$).  
\end{proof}

\subsection{Proof of Proposition~\ref{prop:occ-stability}}\label{sec:proofs-stability}

We give the argument in detail for matchings; the case of independent sets and an analogue for colorings is discussed in Section~\ref{sec:appendix}. 
We begin with a brief summary of the proof from~\cite{Davies2015} of Theorem~\ref{thm:indOcc}.

To prove Theorem~\ref{thm:indOcc} for matchings, let $G$ be any $n$-vertex $d$-regular graph and consider the following experiment. 
We write $N(e)$ for the set of edges incident to an edge $e$ in $G$.  Sample a matching $\mathbf M$ from the monomer-dimer model at fugacity $\lam$ on $G$, and sample an edge $e$ uniformly at random from $E(G)$. 
Then record the \emph{local view} $L$ which consists of the subgraph of $G$ given by the edge $e$ and any incident edges which are \emph{externally uncovered} in the sense that they are not incident to any edge in $\mathbf M\setminus\big(\{e\}\cup N(e)\big)$. 
We write $\mathcal L_d$ for the set of possible local views in $d$-regular graphs.
A graph $G$ and fugacity parameter $\lam$ induce a distribution $p$ on $\mathcal L_d$  such that each $p(L)$ is the probability that the two-part experiment above yields the local view $L$.   

The occupancy fraction of $G$ can be computed from its corresponding distribution $p$. In~\cite{Davies2015} we derived a series of linear constraints on the values $p(L), L \in \mathcal L_d$, that hold for any $d$-regular $G$.  We then relaxed the maximization problem from distributions arising from graphs to all probability distributions on $\mathcal L_d$ satisfying these constraints and showed that the linear programming relaxation is tight: that is, the optimum of the linear program, $\alpha^*$, equals $\alpmatch_{K_{d,d}}(\lam)$. 

This linear program can be expressed in  standard form as 
\begin{equation}\label{eq:occ-primal}
\begin{aligned}
\alpha^* = \max \, a^T p \qquad \text{ subject to } Ap \le b \text{ and } p \ge 0\,,
\end{aligned}
\end{equation}
where $A$ is a matrix with rows indexed by constraints and columns indexed by local views, $a$ is a vector indexed by local views, and $b$ is a vector indexed by constraints. 
We will use the fact that $A$, $a$, and $b$ are independent of $n$. 

The distribution $p^*$ corresponding to $K_{d,d}$ is a feasible solution for this linear program, and so $\alpha^* \ge a^T p^* = \alpha_{K_{d,d}}\match(\lam)$.  To show that $\alpha^* = \alpha_{K_{d,d}}\match(\lam)$ we use linear programming duality.

The dual program (with dual variables given by a vector $q$) is 
\begin{equation}\label{eq:occ-dual}
\alpha^* = \min \, b^T q \qquad \text{ subject to $A^Tq \ge a$ and $q \ge 0$}\,,
\end{equation}
where each dual constraint (given by a row of $A^T$) corresponds to a local view $L$. %
If we can find a feasible primal solution $q^*$ so that $b^T q^* =  \alpha_{K_{d,d}}\match(\lam)$ then we have shown that $\alpha^* \le b^T q^*= \alpha_{K_{d,d}}\match(\lam)$, from which Theorem~\ref{thm:indOcc} follows. In~\cite{Davies2015} we computed $q^*$ by solving the dual constraints corresponding to local views in the support of $p^*$ to hold with equality.  We then define the \emph{slack} as a vector $s\in\mathbb R^{\mathcal L_d}$ such that
\begin{align}\label{eq:dualconstraint}
s_L = A^Tq^* - a_L\,.
\end{align}
We define $\mathcal L_d^*$ as the set of local views $L$ so that $s_L=0$.  Complementary slackness tells us that any primal solution $p'$ that achieves the optimal value $\alpha^*$ must have support contained in $\mathcal L_d^*$.   To prove Proposition~\ref{prop:occ-stability} we show quantitatively that the only graphs with distributions supported in $\mathcal L_d^*$ are $K_{d,d}$ (and $H_{d,n}$).

\begin{lemma}\label{lem:extraconstraint}
Let $G$ be a $d$-regular graph, $\lam>0$, and suppose that $L$ is a random local view obtained from the monomer-dimer model $G$ at fugacity $\lam$ as described above.
Then there exists a function $f(d,\lam)>0$ such that
\begin{align}
\Pr \big(  L \notin \mathcal L_d^*  \big)  \geq \delta_\circ(G,K_{d,d})\cdot f(d,\lam)\,.
\end{align}
\end{lemma}

This lemma gives an additional constraint on the distribution of local views for graphs $G$ with $\delta_\circ(G,K_{d,d})\ge \delta$, namely that
\begin{equation}\label{eq:extraconstraint}
x^T p \ge \delta f(d,\lam)\,,
\end{equation}
where $x$ is the vector in $\{0,1\}^{\mathcal L_d}$ indicating membership of $\mathcal L_d\setminus \mathcal L_d^*$.

To obtain Proposition~\ref{prop:occ-stability} we augment the primal program~\eqref{eq:occ-primal} with the constraint \eqref{eq:extraconstraint} and consider the new dual, which has an additional variable we denote $\theta$,
\begin{equation}\label{eq:occ-augdual}
\alpha^*_{\mathrm{new}} = \min \, b^T q - \delta f\theta\qquad \text{ subject to $A^Tq - \theta x \ge a$, $q \ge 0$, and $\theta\ge 0$}\,.
\end{equation}
If $G$ is a $d$-regular graph with $\delta_\circ(G,K_{d,d})\ge \delta$, then $\alpmatch_G(\lam)$ is feasible in the augmented primal, hence any $q$ and $\theta$ which are feasible in the new dual yield an upper bound on $\alpmatch_G(\lam)$.
We claim that $q=q^*$ and $\theta=\theta^*=\min\{ s_L : L\in\mathcal L_d\setminus\mathcal L_d^*\}$ are dual-feasible, which is easy to verify from the constraints in~\eqref{eq:occ-augdual}. This follows from the facts that $x$ indicates membership of $L\in\mathcal L_d\setminus\mathcal L_d^*$ and that $\theta^*$ is chosen to minimize the slack $s_L$ over this set.
Then directly from~\eqref{eq:occ-augdual} we obtain Proposition~\ref{prop:occ-stability} for matchings with $c\match(d,\lam)=f(d,\lam)\theta^*>0$.
It remains to prove Lemma~\ref{lem:extraconstraint}, for which we need a simple fact that is used throughout the rest of the paper.  
The fact follows easily from the spatial Markov property, but for completeness we give a direct proof that avoids the technical description of this property.

\begin{fact*}\label{lem:alpbound}
Let $G$ be a graph and let $M$ be a random matching drawn from the monomer-dimer model on $G$ at fugacity $\lam$. Then for any edge $e$ of $G$
\[
\P_{G,\lam}(e\in M)\leq\frac{\lam}{1+\lam}\,.
\]
In particular, $\alpmatch_G(\lam)\leq\frac{\lam}{1+\lam}$.
\end{fact*}
\begin{proof}
Let $\mathcal M_1$ denote the set of matchings in $G$ that contain $e$ and let $\mathcal M_0$ denote the set of matchings that do not. Note that the function $f: \mathcal M_1\to \mathcal M_0$ given by $M\mapsto M-e$ is injective and so
\[
\P(e\in M)=\frac{\sum_{M\in \mathcal M_1}\lam^{|M|}}{\sum_{M\in \mathcal M_0}\lam^{|M|}+\sum_{M\in \mathcal M_1}\lam^{|M|}}\le\frac{\sum_{M\in \mathcal M_1}\lam^{|M|}}{\sum_{M\in f(\mathcal M_1)}\lam^{|M|}+\sum_{M\in \mathcal M_1}\lam^{|M|}}=\frac{\lam}{1+\lam}\,.\qedhere
\]
\end{proof}
We remark that the analogous statement for independent sets, $\pr_{G,\lam}(v\in I) \le \frac{\lam}{1+\lam}$ for any $v \in V(G)$, holds by the same proof.

\begin{proof}[Proof of Lemma~\ref{lem:extraconstraint}]
We refer to~\cite{lovasz2012large} for properties of the sampling distance $\delta_\circ(\cdot,\cdot)$. 
In particular, note that if $\delta_\circ(G,K_{d,d})\ge \delta$ then there exists a radius $r\ge2$ such that the variation distance between the distributions on depth-$r$ neighborhoods obtained by uniform random sampling in $G$ and $K_{d,d}$ is at least $\delta/2$. 
Since $r\ge 2$, the only such neighborhood which occurs in $K_{d,d}$ is $K_{d,d}$ itself, and we obtain that the probability of sampling a different neighborhood in $G$ is at least $\delta$. 
That is, the fraction of vertices of $G$ which are not in a copy of $K_{d,d}$ is at least $\delta_\circ(G,K_{d,d})$. 
Since $G$ is $d$-regular the same fact holds for edges.

Fix an arbitrary $e=u_1u_2 \in E(G')$ which is not in a copy of $K_{d,d}$ and consider the local view generated by choosing a random matching $\mathbf M$ from the monomer-dimer model on $G$ and selecting the edge $e$ as the central edge. We have two cases.

If $e$ is contained in a triangle $T=u_1u_2u_3$, then this triangle is present in the local view at $e$ if and only if the $d-2$ edges outside $T$ which are incident to $u_3$ are not in $\mathbf M$. 
Let $A_e$ be this event. 
In the monomer-dimer model on any graph, an edge is in the random matching with probability at most $\lam/(1+\lam)$, so by successive conditioning $\Pr(A_e) \ge (1+\lam)^{-(d-2)}$. 

If $e$ is not contained in a triangle then, because $e$ is not in a $K_{d,d}$, there must be a pair of distinct vertices $u_3$, $u_4$ in $G$ such that
\begin{enumerate}
\item $e=u_1u_2$, $f=u_2u_3$, and $g=u_3u_4$ are edges of $G$,
\item $u_1u_4$ is not an edge of $G$.
\end{enumerate}
In this case, let $A_e$ be the event that $g$ is the only edge of $G$ at distance $2$ from $e$ in the matching $\mathbf M$.
When no edge incident to $g$ and no edge incident to $e$ are in $\mathbf M$, we have $g\in \mathbf M$ with probability $\lam/(1+\lam)$. Since the number of edges incident to $g$ or at distance $\leq2$ from $e$ is $<2d^2$ we have $\Pr(A_e)\ge \lam(1+\lam)^{-2d^2}$.

The only local views which may arise in $K_{d,d}$ are triangle-free with equal numbers of edges on each side of the `central' edge. 
Then in each of the cases above, the event $A_e$ implies the local view centered on $e$ is not in $\mathcal L_d^*$; for the second case note that $A_e$ implies $f$ is not present but all other edges incident to $e$ are. 

Since $e$ was arbitrary, and at least a fraction $\delta_\circ(G,K_{d,d})$ of the edges of $G$ are not in a copy of $K_{d,d}$, with probability at least $\delta_\circ(G,K_{d,d})\min\big\{(1+\lam)^{-(d-2)}, \lam(1+\lam)^{-2d^2}\big\}$ the random local view $L$ from $G$ is not in $\mathcal L_d^*$.
\end{proof}

\section{From stability to coefficients}
\label{sec:transfer}

To transfer stability results of the form of Theorem~\ref{thm:stability} to sharp bounds on individual coefficients (Theorems~\ref{thm:exactslices}, \ref{thm:girthslices}, and \ref{thm:colorslices}) we use the   probabilistic properties of the relevant statistical physics models.  We use these properties to analyze the coefficients of the partition functions of the optimizing graphs, $H_{d,n}$ for Theorems~\ref{thm:stability} and \ref{thm:colorslices} and $\heawoodn$ for Theorem~\ref{thm:girthslices}.  In these cases the optimizing graph is the union of small isomorphic components.  When we run, say, the monomer-dimer model on e.g.\ $H_{d,n}$, because of the spatial Markov property of the model, the sizes of the intersections of the random matching with each component are i.i.d.\ random variables bounded between $0$ and $d$, and this allows us to use probability theory related to the sums of independent random variables variables.
Note that in the canonical model, the sizes of intersections of the random matching with each component are not independent.

The following is a Local Central Limit Theorem due to Gnedenko~\cite{gnedenko1948local} (see the particular statement below and exposition by Tao~\cite{taoBlog}).
\begin{theorem}[Gnedenko's Local Limit Theorem]
\label{thm:gdedenko}
Let $X_1, \dotsc, X_n$ be i.i.d.\ integer valued random variables with mean $\mu$, variance $\sigma^2$, and suppose that the support of $X_1$ includes two consecutive integers. Let $S_n = X_1 + \dots + X_n$.  Then
\begin{align*}
\pr(S_n =k) &=  \frac{1}{\sqrt{ 2 \pi n} \sigma} \exp \left[ -(k-n\mu)^2/2n \sigma^2   \right ] + o(n^{-1/2})\,,
\end{align*}
with the error term $o(n^{-1/2})$ uniform in $k$.
\end{theorem}
The condition on the support of $X_1$ rules out cases in which, say $X_1$ only takes even values and so the probability $S_n$ is odd is $0$.

We  apply Theorem~\ref{thm:gdedenko} to derive bounds on the relative size of nearby coefficients in the optimizing graphs. 
\begin{lemma}\label{lem:localLim}
For any $\del, \eps >0$, $d, n_1 >0$,  there exists $N$ large enough so that the following is true. Suppose $n> N$ and $\eps n \le k \le (1-\eps)n/2$.  Then there exists $\lam \ge \frac{2\eps}{d}$  so that for all $0 \le r \le n_1$,
\begin{align*}
(1-\del) m_{k-r}(H_{d,n})  & \le \lam^r \cdot  m_k(H_{d,n}) \le  (1+\del) m_{k-r}(H_{d,n}) \,.
\end{align*}
\end{lemma}
\begin{proof}
Choose $\lam$ so that the expected size of the matching drawn from the monomer-dimer model on $H_{d,n}$ is $k$. By our choice of $\lam$, $\frac{\lam}{1+\lam}\geq\alpmatch_{H_{d,n}}(\lam)=\frac{2k}{nd}\geq\frac{2\eps}{d}$ and so $\lam\geq\frac{2\eps}{d}$.  Note that $\alpmatch_{H_{d,n}}(t)=\alpmatch_{K_{d,d}}(t)$ is strictly increasing in $t$ with $\lim_{t\to\infty}\alpmatch_{K_{d,d}}(t)=\frac{1}{d}$. Since $k \le (1-\eps)n/2$ (so that $\alpmatch_{K_{d,d}}(\lam)\le\frac{1-\eps}{d}$) it follows that $\lam$ must be bounded from above by a function of $d$ and $\eps$ (but not $n$).

Let $K= n/2d$ and let $X$ be the size of the random matching drawn from $H_{d,n}$ at fugacity $\lam$. Note that $X$ is the sum of $K$ i.i.d.\ random variables $X_1, \dotsc, X_K$, where $X_1$ is the size of a random matching chosen from $K_{d,d}$ according to the monomer-dimer model at fugacity $\lam$ and so takes the values $\{0, 1, \dots d\}$ with positive probability. Let $\sigma^2$ denote the variance of $X_1$ and note that $\sigma$ is independent of $n$. 
Now Theorem~\ref{thm:gdedenko} gives:
\begin{align*}
\pr(X=k-t) &= \frac{1}{\sqrt{2 \pi K} \sigma} e^{- t^2/2K\sigma^2  } + o(K^{-1/2})\,,
\end{align*}
where hidden constants in the error term depend only on $d$ and $\eps$ since $\lam$ is bounded above and below in terms of $d$ and $\eps$. 

 For $0 \le r \le n_1$ and $K$ large enough as a function of $n_1, \del$, we have 
 \begin{align*}
(1-\del) \pr[X=k-r] &\le \pr[X=k] \le (1+\del) \pr[X=k-r]\,. 
\end{align*}
Since $\pr(X=k-r) = \frac{ m_{k-r}(H_{d,n}) \lam^{k-r} }{\Zmatch_{H_{d,n}}(\lam)}$, we get 
\begin{align*}
(1-\del) m_{k-r}(H_{d,n}) \lam^{-r} & \le  m_k(H_{d,n}) \le  (1+\del) m_{k-r}(H_{d,n})  \lam^{-r}\,,
\end{align*}
and the result follows.
\end{proof}
The same result holds for independent sets in both $H_{d,n}$ and $\heawoodn$.  For colorings the result holds as well, and we state it here for completeness.  
\begin{lemma}\label{lem:ColorlocalLim}
For any $\del, \eps >0$, $d, q, n_1 >0$,  there exists $n_0$ large enough so that the following is true. Suppose $n> n_0$ and $\eps n < k < (1-\eps)dn/2q$.  Then there exists $\lam \le 1 -  g(q,d,\eps)$, so that for all $0 \le r \le n_1$,
\begin{align*}
(1-\del) c^q_{k-r}(H_{d,n})  & \le \lam^r \cdot  c^q_k(H_{d,n}) \le  (1+\del) c^q_{k-r}(H_{d,n}) \, ,
\end{align*}
where $g(q,d,\eps) > 0$.
\end{lemma}
The proof is identical to that of Lemma~\ref{lem:localLim}, we just need the fact that for $k < (1-\eps)dn/2q$, $\lam \le 1- g(q,d,\eps)$ for some $g(q,d,\eps) > 0$.  This holds since $\beta^q_G(1) = 1/q$ for all $G$ and $\beta^q_{K_{d,d}}(\lam)$ is strictly increasing in $\lam$.

\subsection{Matchings of a given size}\label{sec:matchslices}
Here we prove Theorem~\ref{thm:exactslices} for matchings.

Divide $G$ into two parts, $G^\prime$ and $H$, where $H$ is a union of $K_{d,d}$'s and $G^\prime$ contains no copy of $K_{d,d}$.    Let $n_1= | G^\prime|$ and $n_2=n-n_1$. Let $H^\prime = H_{d,n_1}$.    

Given $\eps>0$ and $d$, we choose auxiliary parameters with the following dependences:
\begin{itemize}
\item $\del = \del(d)$.
\item $N_1 = N_1(d,\eps)$.
\item  $\eps ' = \eps'(N_1,\del,d)$.
\item $\del' = \del'(\eps, d)$.
\item $N = N(N_1, \eps', \del, \del', d  )$. 
\end{itemize}

We break our argument into cases depending on whether $G'$ is large or small. If $G'$ is small the intuition is that when uniformly choosing a matching of size $k$ from the canonical model on $G$, the distribution of the matching induced on $G'$ looks as though it has come from the grand canonical model with a particular $\lambda$. This will allow us to use Theorem~\ref{thm:stability}.  If $G'$ is large then we can use Theorem~\ref{thm:stability} directly in combination with the local limit theorem.
We break the case where $G'$ is small further depending on the size of $k$.
If $k$ is very large Lemma~\ref{lem:localLim} does not apply and so we use a slightly different approach there.

\begin{description} 
\item[Small-1] If $n_1$ is small ($n_1\le N_1$) and  $k$ is large (at least $(1- \eps^\prime)n/2$) we show that only the top coefficient of the matching polynomial of $G'$ matters, and that there is a constant factor gap between the top coefficient of $G^\prime $ and $H^\prime$ (regardless of how small $n_1$ is).
\item[Small-2]   If $n_1$ is small and $k$ is moderate, $\eps n < k < (1 - \eps^\prime)n/2$, we show that a random matching of size $k$ in $G$ approximately induces the monomer-dimer model on $H$ with an appropriate $\lam$, and thus we can simply compare the partition functions of $G^\prime$ and $H^\prime$.  This uses the local limit theorem.
\item[Large]  If $n_1$ is large ($n_1 > N_1$) then we use stability (Theorem~\ref{thm:stability}) and the local limit theorem (Theorem~\ref{thm:gdedenko}).
\end{description}

\begin{proof}[\textbf{Small-1:} $n_1 \leq N_1$ and $k \geq(1- \eps^\prime)n/2$]
\hfill

Let $\del =\del(d)>0$ be such that for any $G'$ (with $|G'|=n_1$ divisible by $2d$) containing no copies of $K_{d,d}$ we have 
\begin{equation}
\label{eq:matchTopgap}
 m_{n_1/2}(G^\prime) \le (1-\del) m_{n_1/2}(H_{d,n_1})\,.
 \end{equation}
Such a $\del$ exists by Corollary~\ref{cor:bregmanStable}.
We choose $\eps' =\eps'(\del,d,N_1)$ small enough and $N = N(\eps',N_1)$ large enough so that the following holds. For $n> N - N_1$, $k > (1-\eps^\prime)n/2 $,
\begin{equation}
\label{eqEpPrime}
 \frac{m_{k+1}(H_{d,n})}{m_k(H_{d,n})} < \del \cdot 2^{-d N_1/2}\,.
\end{equation} 
To see this is possible, note that given a matching $M$ of size $k+1$ in a graph, we can obtain $k+1$ matchings of size $k$ by removing one of the edges of $M$. On the other hand given a matching $M'$ of size $k$, there are at most $d(n-2k)/2$ edges which can be added to $M'$ to obtain a matching of size $k+1$. It follows that $\frac{m_{k+1}(H_{d,n})}{m_k(H_{d,n})} \leq \frac{d(n-2k)}{2(k+1)} < 2d \eps' $ and so choosing $\eps' = \frac{\delta \cdot 2^{-d N_1/2}}{2d}$ will suffice.
We then see that
\begin{align*}
m_k(G) &= \sum_{r=0}^{n_1/2} m_r(G^\prime) m_{k-r}(H) \\
&\le m_{n_1/2}(G^\prime) m_{k-n_1/2}(H) + 2^{dn_1/2} \max_{k^\prime > k-n_1/2} m_{k^\prime}(H)  \\
&\le (1-\del) m_{n_1/2}(H^\prime) m_{k-n_1/2}(H) + 2^{dn_1/2} \max_{k^\prime > k-n_1/2} m_{k^\prime}(H) \\
&\le (1-\del) m_{n_1/2}(H^\prime) m_{k-n_1/2}(H) + \del \cdot m_{k-n_1/2}(H) \\
&\le  m_{n_1/2}(H^\prime) m_{k-n_1/2}(H) \\
&\le m_k(H_{d,n})\,,
\end{align*}
where the second inequality follows from~\eqref{eq:matchTopgap} and the third from~\eqref{eqEpPrime}. 
\end{proof}

\begin{proof}[\textbf{Small-2:} $n_1 \leq N_1$ and $k \leq(1- \eps^\prime)n/2$]
\hfill

Choose $0<\del' = \del'(\eps,  d) <1/2$ small enough that for $\lam \ge \frac{2\eps}{d}$, and any $d$-regular $G^\prime $ not containing a copy of $K_{d,d}$, we have $\Zmatch_{G'}(\lam) \le \frac{1-\del'}{1+2\del'} \Zmatch_{H_{d,|G'|}}(\lam)$. Such a $\del'$ exists by Theorem~\ref{thm:stability}.  Choose $N = N(N_1, \eps, \eps')$ large enough that by Lemma~\ref{lem:localLim}, for $n\ge N$, $\eps n< k <(1-\eps' )n/2$, $n_1 \le N_1$ and $0\le r \le n_1$ we have
\begin{equation}
\label{eqsmall2bound}
(1-\del') m_{k-r}(H_{d,n_2})   \le \lam^r \cdot  m_k(H_{d,n_2}) \le  (1+\del') m_{k-r}(H_{d,n_2}) 
\end{equation}
for some $\lam \ge \frac{2\eps}{d}$.  Now we bound
\begin{align*}
m_k(G) &= \sum_{r=0}^{n_1/2}m_r(G')m_{k-r}(H)\\
&\leq (1+2\del')  \sum_{r=0}^{n_1/2}m_r(G')\lam^r m_k(H)\quad \text{by \eqref{eqsmall2bound}} \\
&= (1+2\del') m_k(H)\Zmatch_{G'}(\lam) \\
&\le (1-\del') m_k(H)\Zmatch_{H'}(\lam) \\
&= (1 - \del') \sum_{r=0}^{n_1/2}m_r(H')\lam^r m_k(H) \\
&\le \sum_{r=0}^{n_1/2}m_r(H') m_{k-r}(H)\quad \text{by \eqref{eqsmall2bound}} \\
&= m_k(H_{d,n})\,.\qedhere
\end{align*}
\end{proof}

\begin{proof}[\textbf{Large:} $n_1>N_1$]
\hfill

Choose $\lam$ so that the expected size of the random matching in $H_{d,n}$ at fugacity $\lam$ is $k$, and choose $s$ so that
\[
 m_s(G') m_{k-s}(H) = \max_{0\le r \le n_1/2} m_r(G') m_{k-r}(H)\, .
\]

In other words, $s$ is the most likely size of matching induced on $G'$ when choosing a random matching from $G$ of size exactly $k$ (note that the choice of $s$ may not be unique).  
We write $\pr_{G,\lam}[|\mathbf M|= t]$ to mean the probability that a random matching from the monomer-dimer model on $G$ at fugacity $\lam$ has $t$ edges. 
Then we have 
\begin{align*}
m_k(G) &\le  \left(\frac{n_1}{2}+1\right)  \cdot m_s(G') m_{k-s}(H) \\
&=  \left(\frac{n_1}{2}+1\right) \cdot \frac {\Zmatch_{G'}(\lam)}{\lam^s} \pr_{G',\lam} [| \mathbf M | = s] \cdot \frac {\Zmatch_{H_{d,n_2}}(\lam)}{\lam^{k-s}} \pr_{H_{d,n_2},\lam} [|\mathbf M|= k-s] \, ,
\end{align*}
and
\begin{align*}
m_k(H_{d,n}) &\ge  m_{\lfloor kn_1/n \rfloor}(H') m_{\lceil kn_2/n \rceil}(H)  \\
&=   \frac{\Zmatch_{H_{d,n_1}}(\lam)}{\lam^{\lfloor kn_1/n \rfloor}} \pr_{H_{d,n_1},\lam} [|\mathbf M| = \lfloor kn_1/n \rfloor] \cdot \frac{\Zmatch_{H_{d,n_2}}(\lam)}{\lam^{\lceil kn_2/n \rceil}} \pr_{H_{d,n_2},\lam} [| \mathbf M| = \lceil kn_2/n \rceil] \, .
\end{align*}
Cancelling terms, it is enough to show that 
\begin{equation}
\label{eqlargeMatchZ}
\frac{\Zmatch_{H_{d,n_1}}(\lam)}{ \Zmatch_{G'}(\lam) }  \cdot \frac{ \pr_{H_{d,n_1},\lam} [|\mathbf M|= \lfloor kn_1/n \rfloor]}{  \pr_{G',\lam} [|\mathbf M| = s] } \cdot  \frac{\pr_{H_{d,n_2},\lam} [| \mathbf M | = \lceil kn_2/n \rceil]   }{    \pr_{H_{d,n_2},\lam} [|\mathbf M| =k- s]   }   \ge  \left(\frac{n_1}{2}+1\right) \,.
\end{equation}
 Using Theorem~\ref{thm:gdedenko} (Gnedenko's Local Limit Theorem)
, we see that  since $\lceil kn_2/n \rceil $ is within $1$ of the mean size of matching in $H_{d,n_2}$ at fugacity $\lam$
, we have 
\[
\pr_{H_{d,n_2},\lam} [| \mathbf M | = \lceil kn_2/n \rceil] \geq \Omega(1/\sqrt{n_2})\,,
\]
and 
\[ \pr_{H_{d,n_2},\lam} [| \mathbf M | =k- s] \leq O(1/\sqrt{n_2})\,.
\]
Therefore the factor $ \frac{\pr_{H_{d,n_2},\lam} [| \mathbf M | = \lceil kn_2/n \rceil]   }{    \pr_{H_{d,n_2},\lam} [| \mathbf M | =k- s]   } $ is lower bounded by a constant depending only on $\eps$ and $d$, independent of $n$ and $n_1$, call this $\eta(\eps, d)$. Likewise by Theorem~\ref{thm:gdedenko} again since $\lfloor kn_1/n \rfloor $ is within $1$ of the mean size  of the random matching in $H'$ at fugacity $\lam$, the factor   $\pr_{H_{d,n_1},\lam} [| \mathbf M | = \lfloor kn_1/n \rfloor]$ is lower bounded by $1/ \sqrt{n_1}$ times another constant $\eta'(\eps,d)$.  Finally, $ \pr_{G',\lam} [| \mathbf M| = s]  \le 1$.   Using Theorem~\ref{thm:stability}, we can choose $N_1= N_1(\eps,d)$ large enough that for all $n_1>N_1$,  
\begin{align*}
\frac{\Zmatch_{H_{d,n_1}}(\lam)}{ \Zmatch_{G'}(\lam) } &\ge  \left(\frac{n_1}{2}+1\right) \cdot \frac{\sqrt{n_1}}{\eta(\eps,d) \cdot \eta'(\eps,d) } \, ,
\end{align*}
which gives~\eqref{eqlargeMatchZ}.
\end{proof}

\section{A hierarchy of extremal results for partition functions and their coefficients}
\label{sec:hierarchy}

In this section we investigate several possible ways in which one partition function can dominate  another. 
We discuss notions of dominance that correspond to extremal results in the literature, and prove general implications between them. 

In Section~\ref{sec:intro} we discussed dominance in terms of counting (evaluating the partition function at $\lam=1$); evaluating the partition function at any $\lam$ (Theorems~\ref{thm:matchPart} and~\ref{thm:indPart}); in terms of individual coefficients (Conjectures~\ref{conj:IndGivenSize} and~\ref{conj:UMC}); and in terms of the highest degree coefficient (Bregman's Theorem~\cite{bregman1973some}).  In Section~\ref{sec:probability} we discussed dominance in terms of the logarithmic derivative or occupancy fraction (Theorems~\ref{thm:indOcc} and~\ref{thm:colorOcc}).  

Another notion of dominance relates to the canonical ensemble. 
The occupancy fractions and internal energy are not interesting observables in the canonical ensemble; they are fixed at $k/|V(G)|$ and $k/|E(G)|$ respectively by definition.  Instead we consider the \textit{free volume}.  For an independent set $I$ of a graph $G$ we let $F\ind_G(I)$ denote the set of vertices of $G$ that are neither in $I$ nor adjacent to $I$.  That is, they are precisely the vertices $v$ which can be added to $I$ to make a larger independent set.  This is the free volume of $I$ in $G$.  For a matching $M$ we let $F\match_G(M)$ denote the set of edges of $G$ that are neither in $M$ nor incident to an edge of $M$.

Now consider the expected free volume when choosing a random independent set or matching from the corresponding canonical ensemble; we denote these by
\[f\ind_k(G) = \E |F\ind_{G}(\mathbf I_k)|\,,\]
\[f\match_k(G) = \E |F\match_{G}(\mathbf M_k)|\,.\]
Observe that we can write $f\ind_k(G) $ and $f\match_k(G)$ in terms of the coefficients of the respective partition functions:
\[f\ind_k(G) = (k+1) \frac{ i_{k+1}(G)  }{  i_k(G)  } \,,\]
\[f\match_k(G) = (k+1) \frac{ m_{k+1}(G)  }{  m_k(G)  } \,.\]
We can ask for dominance of the expected free volume in the canonical ensemble, that is, showing that $f_k^{\mathrm{ind}}(G) = (k+1)\frac{i_{k+1}(G)}{i_k(G)}$ is maximized for all $k$ by some particular extremal graph.  
In~\cite{Davies2015} we conjectured $H_{d,n}$ has this property for matchings and independents sets in a regular graphs.

\begin{conj}[\cite{Davies2015}]
\label{conj:FVmax}
Let $2d$ divide $n$. Then for all $k$; over $d$-regular, $n$-vertex graphs $G$, the quantities $f\ind_k(G)$ and $f\match_k(G)$ are maximized by $H_{d,n}$.
\end{conj}
As we show below in Proposition~\ref{thm:hierarchy}, dominance of the expected free volume in the canonical model implies all of the other notions of dominance of partition functions and their coefficients mentioned above. The notion was already investigated implicitly for lower bounds in the hard-core model by Cutler and Radcliffe~\cite{cutler2014maximum}.

\begin{theorem}[Cutler, Radcliffe \cite{cutler2014maximum}]
\label{thm:CutRad}
Let $d+1$ divide $n$. Then for any $d$-regular $G$ on $n$ vertices, and any $1 \le k \le n/(d+1)$, 
\begin{equation}
\label{eq:cutlerRad}
 f_k^{\mathrm{ind}}(G) \ge f_k^{\mathrm{ind}}(\CLdn),
 \end{equation}
where $\CLdn$ is the union of $n/(d+1)$ disjoint cliques on $d+1$ vertices. 
As a consequence, for all $k$ we have 
\[  i_k(G) \ge   i_k (\CLdn)  \,.\]
\end{theorem}
In fact, the free volume interpretation gives a very short proof of the theorem: in a $d$-regular graph each vertex in an independent set covers $d$ other vertices; in a union of cliques the sets covered by each vertex in an independent are disjoint, and so $\CLdn$ maximizes the expected number of covered vertices in the canonical model, or equivalently, it minimizes the expected free volume.

In the next section we describe the rich picture of implications between the different notions of dominance described above.

\subsection{A hierarchy of counting theorems}

Let $Z_G(\lam) = \sum_{k=0}^n c_k(G) \lam^k $ and $Z_H(\lam) = \sum_{k=0}^n c_k(H) \lam^k $, and suppose $c_0(G) =c_0(H)$ and that $c_k(G), c_k(H) \ge 0$ for all $k$. 

Consider the following statements:
\begin{description}
\item[COUNT] $\sum_{k=0}^n c_k(G) \ge \sum_{k=0}^n c_k(H)$.
\item[PART] $Z_G(\lam) \ge Z_H(\lam)$ for all $\lam \ge 0$.
\item[COEF] $c_k(G) \ge c_k(H)$ for all $1 \le k \le n$.
\item[OCC] $\frac{\lam Z_G'(\lam)}{Z_G(\lam)} \ge \frac{\lam Z_H'(\lam)}{Z_H(\lam)}$ for all~$\lam >0$.
\item[MAX] $c_n(G) \ge c_n(H)$.
\item[FV] $\frac{c_{k+1}(G)}{c_k(G)} \ge \frac{c_{k+1}(H)}{c_k(H)}$ for all $0 \le k \le n-1$.
\end{description}

For the independence polynomial of a regular graph, the theorems of Kahn, Galvin and Tetali, Zhao, and Theorem~\ref{thm:indOcc} state that \textbf{COUNT}, \textbf{PART}, and \textbf{OCC} hold with  $ G=H_{d,n}$ and $H$ any $d$-regular graph on $n$ vertices; the statements \textbf{COEF} and \textbf{FV} are Conjectures~\ref{conj:IndGivenSize} and \ref{conj:FVmax}. 

For the matching polynomial of a regular graph, Bregman's theorem is \textbf{MAX}, while Theorem~\ref{thm:indOcc} is \textbf{OCC}.
By the following proposition this implies \textbf{COUNT} and \textbf{PART} and provides an alternative proof of \textbf{MAX}.   The statements \textbf{COEF} and \textbf{FV} are Conjectures~\ref{conj:UMC} and \ref{conj:FVmax}.

\begin{prop}
\label{thm:hierarchy}
Let $Z_G(\lam)$, $Z_H(\lam)$ be defined as above. Then
\begin{enumerate}
\item {\bf PART} implies {\bf COUNT}.
\item {\bf PART} implies {\bf MAX}.
\item {\bf COEF} implies {\bf PART}. 
\item {\bf OCC} implies {\bf PART}.\label{itm:occ-part}
\item {\bf FV} implies {\bf COEF} and {\bf OCC}.
\item {\bf COUNT} and {\bf MAX} are incomparable in general.
\item  {\bf COEF} and {\bf OCC} are incomparable in general. 
\end{enumerate}
\end{prop}
  \begin{figure}
\centering
\caption{The implications of Proposition~\ref{thm:hierarchy}}
\includegraphics[scale=0.25]{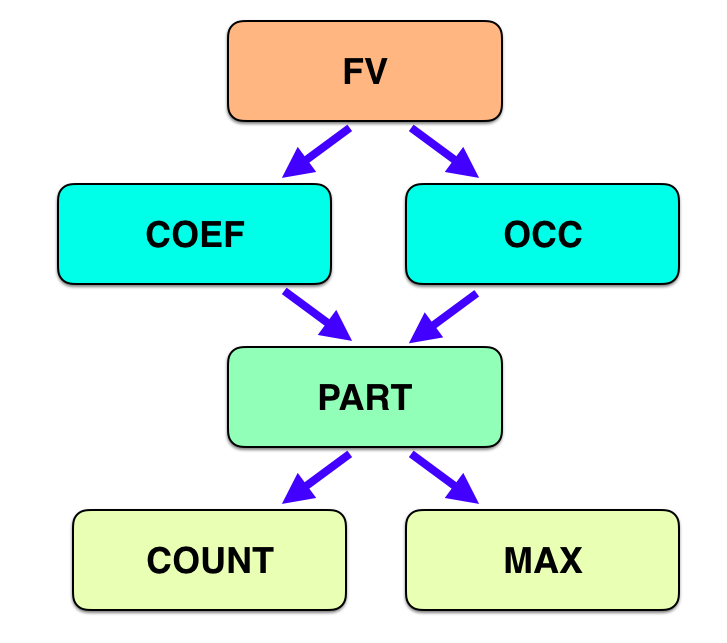}
\end{figure}

\begin{proof}\noindent
\begin{enumerate}
\item Immediate by taking $\lam =1$.
\item Take $\lam \to \infty$.  
\item Immediate since the coefficients are all non-negative. 
\item We calculate
\begin{align*}
 \log Z_G(\lam) &= \log Z_G(0) +  \int_{0}^\lam (\log Z_G(t))' \, dt \\
 &= \log c_0(G) +  \int_{0}^\lam \frac{ Z_G'(t)  }{ Z_G(t)   } \, dt \\
&\ge \log c_0(H) +  \int_{0}^\lam \frac{ Z_H'(t)  }{ Z_H(t)   } \, dt \\
 &=\log Z_H(0) +  \int_{0}^\lam (\log Z_H(t))' \, dt \\
 &= \log Z_H(\lam). 
\end{align*}
\item {\bf FV} implies {\bf COEF} by induction on $k$.

To show {\bf FV} implies {\bf OCC}, we need to show
\begin{align*}
\left( \sum_{j=1}^n j c_j(G) \lam^{j-1}  \right)  \left( \sum_{j=0}^n c_j(H) \lam^j  \right)  -  \left( \sum_{j=1}^n j c_j(H) \lam^{j-1}  \right)  \left( \sum_{j=0}^n c_j(G) \lam^j  \right)&\ge0\,.
\end{align*}
Writing the left hand side as $ \sum_{j=0}^n r_j \lam^{j} $, we will in fact show that $r_j\ge0$ for all $j$.  
Let $a_i=c_i(G)$ and $b_i=c_i(H)$. 
We have $r_0=a_1b_0- b_1a_0$  and so $r_0\ge0$ by \textbf{FV} and $a_0=b_0$.  Similarly we have $r_1=a_2b_0 -b_2a_0$ and
so we see that $r_1\ge0$ by multiplying two cases of \textbf{FV}. Continuing, we have 
\begin{align*}
r_2=3 a_3b_0  +  a_2 b_1   -3 b_3a_0  -  a_1 b_2 
\end{align*}
and so $r_2\ge 0$ since $a_0=b_0$, $a_3 \ge b_3$, and $a_2/a_1 \ge b_2/b_1$. 

In general we have
\begin{align*}
r_k&=\sum_{j= 0}^k (k-j+1) b_j a_{k-j+1}  -  \sum_{j=0}^k   (k-j+1) a_j b_{k-j+1} \\
&= \sum_{j= 0}^k (k-j+1) b_j a_{k-j+1} - \sum_{j=1}^{k+1} j a_{k-j+1}   b_j \\
&= \sum_{j= 0}^{k+1} (k-j+1) b_j a_{k-j+1} - \sum_{j=0}^{k+1} j a_{k-j+1}   b_j \\
&= \sum_{j= 0}^{k+1} (k-2j+1) b_j a_{k-j+1}  \\
&= \sum_{j=0}^{\lfloor (k+1)/2 \rfloor} (k-2j+1) b_j a_{k-j+1} -\sum_{j=0}^{\lfloor (k+1)/2 \rfloor} (k-2j+1) a_j b_{k-j+1}\, .
\end{align*}
Now considering the last line term by term, we see the inequality $r_k\ge0$ reduces to showing that 
\[ \frac{ a_{k-j+1}  }{ a_j  }  \ge  \frac{ b_{k-j+1}  }{ b_j  }   \]
for $j \le \lfloor (k+1)/2 \rfloor$. This follows by multiplying successive cases of {\bf FV}.

\item {\bf COUNT} and {\bf MAX} are incomparable: take, for example, $Z_G(\lam) = 1+ 5\lam +2 \lam^2$ and $Z_H(\lam) = 1+ 2 \lam + 3 \lam^2$. 

\item {\bf COEF} and {\bf OCC} are incomparable:  take first $Z_G(\lam)=1+ 2\lam+\lam^2$ and $Z_H(\lam) =1+3\lam+\lam^2$.  Then $Z_H$ dominates $Z_G$ coefficient by coefficient, but for large $\lam$ its logarithmic derivative is smaller.    Next take $Z_G(\lam) = 1+ 5  \lam + 5 \lam^2 + 5\lam^3$ and $Z_H(\lam)= 1+ 4 \lam +6 \lam^2 +\lam^3$.  Then $Z_G(\lam)$ has a larger logarithmic derivative at all $\lam$ but does not dominate $Z_H(\lam)$ coefficient by coefficient.  \qedhere    
\end{enumerate}
\end{proof}

In general while {\bf COEF} implies {\bf PART}, the converse does not hold and so Conjectures~\ref{conj:IndGivenSize} and~\ref{conj:UMC} remain open.  Nevertheless,  the method of Section~\ref{sec:transfer} gives a partial converse: that with uniqueness and stability, {\bf PART} implies {\bf COEF} for large enough graphs and almost all values of $k$.

\section{Proofs for other models}\label{sec:appendix}

In this section we complete the proofs of Theorems~\ref{thm:exactslices} and~\ref{thm:stability}, and prove Theorems~\ref{thm:girthslices} and~\ref{thm:colorslices}. 

\subsection{Independent sets}

We first discuss the modifications of the arguments of Sections~\ref{sec:stable} and~\ref{sec:transfer} necessary to obtain Theorems~\ref{thm:exactslices} and~\ref{thm:stability} for independent sets in $d$-regular graphs, and the analogous Theorem~\ref{thm:girthslices} for cubic graphs of girth at least $5$.

To adapt Section~\ref{sec:stable} it suffices to prove Proposition~\ref{prop:occ-stability} for independent sets, and the first part of the following result for independent sets in cubic graphs of girth at least $5$. 
The derivation of stability for the partition function is identical to the case of matchings, see the proof of Theorem~\ref{thm:stability} in Section~\ref{sec:deducestability}. 

\begin{prop}\label{prop:girth-stability}
There exists a continuous function $c^{\mathrm{girth}}(\lam)$ which is strictly positive when $\lam>0$ so that for any $3$-regular graph $G$ of girth at least $5$, and any $\lam \ge 0$,
\begin{align*}
 \alpind_G(\lam) &\le \alpind_{\heawood}(\lam) - c^{\mathrm{girth}}(\lam) \cdot\delta_{\circ}(G,\heawood)\,,
\end{align*}
As a consequence, there is a function $s^{\mathrm{girth}}(\lam)$, strictly increasing in $\lam$ with $s^{\mathrm{girth}}(0)=0$, so that for any $3$-regular graph $G$ and any $\lam \ge 0$,
\begin{align*}
\frac{1}{|V(G)|} \log \Zind_G(\lam) & \le \frac{1}{14} \log \Zind_{\heawood}(\lam) - s^{\mathrm{girth}}(\lam) \cdot \del_{\circ}(G,\heawood).
\end{align*}
\end{prop} 

The manipulation of the linear programs from~\cite{Davies2015,perarnau2016counting} which show that distributions arising in $K_{d,d}$ and $\heawood$ (or unions thereof) are the unique maximizers of $\alpind_G(\lam)$ over the relevant graph classes is the same as in the proof of Proposition~\ref{prop:occ-stability}, hence it suffices to give an analogue of Lemma~\ref{lem:extraconstraint} for each case. 
For the case of $3$-regular graphs of girth at least $5$ there is a minor additional consideration; the sets consisting of local views which occur in $\heawood$ and of local views with slack equal to zero do not coincide. In the proof we work with the latter set, so the methods are identical. 

For a vertex $v$ in a graph $G$, write $N_t(v)$ for the set of vertices at distance exactly $t$ from $v$ in $G$, and $\overline N_t(v)$ for the vertices at distance at most $t$ from $v$ in $G$. 
We will consider the subgraphs graph of $G$ induced by these sets, denoted e.g.\ $G[N_t(v)]$.

\begin{lemma}
Let $G$ be a $d$-regular graph and $\lam>0$. 
Let $L$ be the random local view obtained by drawing an independent set $\mathbf I$ from the hard-core model on $G$ at fugacity $\lam$, drawing a vertex $v$ in $G$ uniformly at random, and setting $L$ to be the induced subgraph of $G$ on the neighbors of $v$ which are not adjacent to any vertex in $\mathbf I\setminus N(v)$. 
Write $\mathcal L_d^*=\{ \emptyset, \overline{K_d}\}$ for the set of local views which may arise in $K_{d,d}$ (and which are exactly the local views which have zero slack in the dual linear program for independent sets from~\cite{Davies2015}). 
Then
\begin{align}
\Pr \big[  L \notin \mathcal L_d^*  \big]  &> \delta_\circ(G,K_{d,d})\cdot\frac{\lam}{(1+\lam)^{2d+1}}\,.
\end{align}
\end{lemma}

\begin{proof}
Fix an arbitrary $v \in V(G)$ which is not in a copy of $K_{d,d}$, and let $u_1, u_2$ both be neighbors of $v$ so that $N(u_1) \ne N(u_2)$. 
Such $u_1, u_2$ must exist otherwise $v$ is in a copy of $K_{d,d}$. 
Let $w$ be a vertex neighboring $u_1$ but not $u_2$ and write $A_v$ for the event that $\{ \mathbf I \cap N(u_2)=\emptyset \} \cap \{ w \in \mathbf I \}$, where $\mathbf I$ is a random independent set from the hard-core model on $G$.
Note that $A_v$ implies that the local view rooted at $v$ cannot arise in $K_{d,d}$ because $u_2$ is uncovered while $u_1$ is not.

By successive conditioning, the probability that $\mathbf I \cap \{ N(u_2) \cup N(w) \} = \emptyset$ is at least $(1+\lam)^{-2d}$. 
Now conditioned on $I \cap N(w) = \emptyset$, $\Pr[w \in I] = \frac{\lam}{1+ \lam}$, and so $\Pr [A_v] \ge \frac{\lam}{(1+\lam)^{2d+1}}$.
Since $v$ was arbitrary, and at least a fraction $\delta_\circ(G,K_{d,d})$ of the vertices in $V(G)$ are not in a copy of $K_{d,d}$, the required bound follows.
\end{proof}

\begin{lemma}
Let $G$ be a cubic graph of girth at least $5$ and $\lam>0$. 
Let $L$ be the random local view obtained by drawing an independent set $\mathbf I$ from the hard-core model on $G$ at fugacity $\lam$, drawing a vertex $v$ in $G$ uniformly at random, and setting $L$ to be the induced subgraph of $G$ on $\overline N_2(v)$ which are not adjacent to any vertex in $\mathbf I\setminus \overline N_2(v)$.
Write $\mathcal L^*$ for the set of such local views which have zero slack in the dual linear program from~\cite{perarnau2016counting}. 
Then
\begin{align}
\Pr \big[  L \notin \mathcal L^*  \big]  &> \delta_\circ(G,\heawood)\cdot \min\left\{\frac{\lam}{(1+\lam)^{15}},\frac{1}{(1+\lam)^2}\right\}\,.
\end{align}
\end{lemma}

\begin{proof}
The proof relies on the fact that $\heawood$ the only connected cubic graph of girth at least $5$ containing a vertex $v$ with which is not in a $5$-cycle and which has $\abs{N_3(v)}=4$. To see this note that in such a graph $G$, the subgraph $G[\overline N_2(v)]$ must be a tree with $6$ leaves. 
If $\abs{N_3(v)}=4$ then it is easy to check that there is a unique way to connect $N_3(v)$ to $N_2(v)$ without creating cycles of length at most $4$, such that each vertex has degree $3$, and that this construction gives a copy of $\heawood$.
Now suppose that $v$ is a vertex of $G$ that not contained in a copy of $\heawood$. There are two cases. 

If $v$ is not contained in a $5$-cycle, then by the above fact there must be at least $5$ vertices in $N_3(v)$. 
Hence there exists a vertex $w\in N_3(v)$ which is connected to either $1$ or $2$ vertices in $N_2(v)$. 
Let $A_v$ be the event that $\{w\}=\mathbf I \cap N_3(v)$. 
If $w$ is connected to $2$ vertices in $N_2(v)$ they cannot have second common neighbor, else $G$ would contain a $4$-cycle. 
Hence the event $A_v$ implies that, writing $N_1(v)=\{u_1, u_2, u_3\}$, the sizes of $N_2(v)\cap N_1(u_i)$ in $L$ are $(1,2,2)$ or $(1,1,2)$.
Using the facts that any vertex is in $\mathbf I$ with probability at most $\lam/(1+\lam)$, and that this holds with equality for any $w$ conditioned on $\mathbf I \cap N(w)=\emptyset$, and that $\abs{N_3(v)}\le 12$, we have $\Pr[A_v] \ge \lam(1+\lam)^{-15}$.

If $v$ is contained in a $5$-cycle then let $A_v$ be the event that the $5$-cycle is present in $L$. Then we have $\Pr[A_v]\ge (1+\lam)^{-2}$.

In each case, the event $A_v$ implies that the local view at $v$ is one that has positive slack. 
The lemma follows.
\end{proof}

It is now straightforward to adapt the methods of Section~\ref{sec:matchslices} and prove Theorem~\ref{thm:exactslices} and~\ref{thm:girthslices} for independent sets: essentially identical variants of all the required results are in place.

\subsection{Colorings}

We now turn to our remaining example, colorings of cubic graphs.  Here the proofs are slightly different but follow the same idea.
In proving Theorem~\ref{thm:colorslices}, the main difference is that Theorem~\ref{thm:colorOcc} holds only for $\lam <1$, and  $H_{3,n}$ only maximizes $Z_G^q(\lam)$ in this range. In fact, for $\lam=1$ the free energy is the same for all graphs, so we cannot hope to show that $H_{3,n}$ maximizes all of the individual coefficients of $Z^q_G(\lam)$. 
Instead we can show that it maximizes the `anti-ferromagnetic' coefficients; i.e.\ those for which $k$ is less than the mean number of monochromatic edges in the non-interacting case $\lam=1$: $\E_{G,1}[ m(G, \boldsymbol \chi)] =\frac{3n}{2q}$  for any $3$-regular graph $G$.  Thus Theorem~\ref{thm:colorslices} holds for all $0 \le k \le (1-\eps)3n/2q$. 

The results for colorings which correspond to those of Section~\ref{sec:stable} are as follows.

\begin{prop}\label{prop:color-stability}
There exists a continuous function $c^q(\lam)$ which is strictly positive when $\lam \in (0,1)$ such that for any $3$-regular graph $G$, any $q \ge 2$, and $\lam \in [0,1]$,
\begin{align*}
 \beta^q_G(\lam) &\ge \beta^q_{K_{3,3}}(\lam) + c^q(\lam) \cdot\delta_{\circ}(G,K_{3,3})\,.
\end{align*}
As a consequence, there is a function $s^q(\lam)$, strictly decreasing in $\lam$ with $s^q(1)=0$, so that for any $3$-regular graph $G$, any $q \ge 2$, and $\lam \in [0,1]$,
\begin{align*}
\frac{1}{|V(G)|} \log Z^q_G(\lam) & \le \frac{1}{6} \log Z^q_{K_{3,3}}(\lam) - s^q(\lam) \cdot \del_{\circ}(G,K_{3,3})\,.
\end{align*}
\end{prop}

To prove Proposition~\ref{prop:color-stability} we require an analogue of Lemma~\ref{lem:extraconstraint}. 

\begin{lemma}\label{lem:extracolconst}
Let $G$ be a $d$-regular graph, $q\ge 2$, $0\le\lam<1$, and let $L$ be the random local view obtained by drawing a random $q$-coloring $\boldsymbol\chi$ from the $q$-color Potts model on $G$, sampling a uniformly random vertex $v$ in $G$, and setting $L$ to be the $G[\overline N_2(v)]$ together with the vertices of $N_2(v)$ colored by $\boldsymbol\chi$. Write $\mathcal L_d^*$ for the set of such local views which have zero slack in the dual linear program from~\cite{davies2016extremes}. Then there exists $f(q,d,\lam)>0$ such that
\begin{align}
\Pr \big[ L \notin \mathcal L_d^* \big]  &> \delta_\circ(G,K_{d,d})\cdot f(q,d,\lam)\,.
\end{align}
\end{lemma}
\begin{proof}
We show that there is a constant $f'(q,d,\lam)>0$ such that every vertex not in a copy of $K_{d,d}$ has a local view not in $\mathcal L_d^* $ with probability at least $f'(q,d,\lam)$.

For any vertex $v$ in $K_{d,d}$ (or $H_{d,n}$) any two neighbors of $v$ have the same neighborhood as each other.
Unions of complete bipartite graphs are the only graphs with this property.
If $v$ is a vertex not in a copy of $K_{d,d}$ it must have two neighbors ($u$ and $w$) with distinct neighborhoods.
We argue that with probability bounded below by a function of $q$, $d$ and $\lam$ these two neighbors of $v$ have neighborhoods with distinct colorings.
It is clear that the probability that the neighborhoods of $u$ and $w$ have different colorings is positive since every coloring is possible.
We must show that there is a lower bound independent of $n$.
For any set of $t$ vertices every choice of coloring of these vertices happens with probability at least $(\lam^d/q)^{t}$.
This is because there are at most $d \cdot t$ edges incident vertices in the set, so the contribution to the partition function is at least $\lam ^{d \cdot t}$ and the partition function is at most $q^t$ whenever $\lam \leq 1$.
From this we see that there is at least a $(\lam^d/q)^{2d}$ probability that the neighborhoods of $u$ and $w$ are colored differently.
\end{proof}

\begin{proof}[Proof of Proposition~\ref{prop:color-stability}]
As with independent sets, the manipulation of the linear program from~\cite{davies2016extremes} with the additional constraint offered by Lemma~\ref{lem:extracolconst} is the same as in Section~\ref{sec:proofs-stability}. The required statement for the internal energy follows.

We now show how the partition function result follows from the stability of the internal energy.  For $\lam \in [0,1]$, we have:
\begin{align*}
\frac{1}{|V(G)|} \log Z^q_G(\lam) &= \frac{1}{|V(G)|} \log Z^q_G(1) - \frac{d}{2}\cdot\int_{\lam}^1 \frac{\beta^q_G(t)}{t} \, dt \\
&= \log q -  \frac{d}{2}\cdot\int_{\lam}^1 \frac{\beta^q_G(t)}{t} \, dt \\
&\le \log q  - \frac{d}{2}\cdot\int_{\lam}^1\left( \frac{\beta^q_{K_{3,3}}(t)}{t} + \frac{c^q(t) \cdot\delta_{\circ}(G,K_{3,3})}{t}\right) \, dt  \\
&= \frac{1}{6} \log Z^q_{K_{3,3}}(\lam) - s^q(\lam) \cdot\delta_{\circ}(G,K_{3,3}) \,,
\end{align*}
where we can set $s^q(\lam) = \frac{d}{2}\int_{\lam}^1 \frac{c^q(t) }{t} \, dt$. 
\end{proof}
As a corollary of Proposition~\ref{prop:color-stability} we obtain the following stability result on the number of proper colorings. 
\begin{cor}\label{cor:stableColor}
For all $q \geq 2$, there exists $ c(q)>0$ such that for any $n$ divisible by $6$ and any $3$-regular graph $G$ on $n$ vertices, we have
\begin{align*}
& \frac{1}{n} \log  c_0^q(G) \leq  \frac{1}{n} \log c_0^q(H_{3,n}) - c(q)\cdot\delta_{\circ}(G,K_{3,3})
\end{align*}
\end{cor}
The proof is identical to that of Corollary~\ref{cor:bregmanStable}.

\subsubsection{Proof of Theorem~\ref{thm:colorslices}}
  Here we keep $d$ as a variable instead of setting $d=3$, to show that if a version of Theorem~\ref{thm:colorOcc} can be proved for $d\ge 4$, then the corresponding result on the individual coefficients will follow. We again split $G$ into two parts, $G'$ of size $n_1$ with no $K_{d,d}$ components, and $H = H_{d,n_2}$, where $n_2 = n-n_1$, consisting of all $K_{d,d}$ components. 
  
  Given $\eps>0$, $d$, and $q$, we choose auxiliary parameters with the following dependence:
\begin{itemize}
\item $\del = \del(d,q)$.
\item $N_1 = N_1(d,q,\eps)$.
\item  $\eps ' = \eps'(N_1,\del,d,q)$.
\item $\del' = \del'(\eps, d,q)$.
\item $N= N(N_1, \eps', \del, \del', d,q  )$. 
\end{itemize}

\begin{proof}[\textbf{Small-1}: $n_1\le N_1$ and $k < \eps^\prime n$]
\hfill

Let $\del = \del(q,d)$ be such that for any $G'$ (with $|G'|=n_1$ divisible by $2d$) containing no copies of $K_{d,d}$ we have 
\begin{equation}
\label{eq:ColorTopgap}
 c^q_{0}(G^\prime) \le (1-\del) c^q_{0}(H_{d,n_1})\,.
 \end{equation}
Such a $\del$ exists by Corollary~\ref{cor:stableColor}.
\begin{claim}
For $q \geq 3$ there exists  $ \eps^\prime$ small enough and $N$ large enough  (as functions of $d,q, \del, N_1$) such that the following holds.
For all $n> N$, $k < \eps^\prime n $,
\begin{equation*}
\label{eqColorEpPrime}
 \frac{c^q_{k-1}(H_{d,n})}{c^q_k(H_{d,n})} < \del \cdot q^{- N_1}\,. 
 \end{equation*} 
\end{claim}
\begin{claimproof}
For each $q$-coloring of $H_{d,n}$ with exactly $k-1$ monochromatic edges we can create a $q$-coloring with exactly $k$ monochromatic edges in the following way.
There are at least $\frac{n}{2d}-k \geq (\frac{1}{2d}-\eps')n\geq \frac{n}{4d}$ copies of $K_{d,d}$ with no monochromatic edges so we choose one of these.
We recolor this copy of $K_{d,d}$ in any way that gives a single monochromatic edge, which is possible whenever $q \geq 3$.
We had at least $n/4d$ choices of the $K_{d,d}$ to recolor and each coloring of $H_{d,n}$ with $k$ monochromatic edges can be created by this method at most $k\cdot q^{2d}$ times as there are $k$ monochromatic edges and at most $q^{2d}$ colorings of the $K_{d,d}$ that we altered.
Thus
\[
c_k^q(H_{d,n}) \geq \frac{n}{4d \cdot k \cdot q^{2d}}\cdot c_{k-1}^q(H_{d,n})\,,
\]
and so the result follows if 
\[
\frac{1}{4d\eps'q^{2d}} > \delta \cdot q^{-N_1}\,.\qedhere
\]
\end{claimproof}
We then see that
\begin{align*}
c^q_k(G) &= \sum_{r=0}^{dn_1/2} c^q_r(G^\prime) c^q_{k-r}(H) \\
&\le c^q_{0}(G^\prime) c^q_{k}(H) + q^{n_1} \max_{k^\prime < k} c^q_{k^\prime}(H)  \\
&\le (1-\del) c^q_{0}(H^\prime) c^q_{k}(H) + q^{n_1} \max_{k^\prime < k} c^q_{k^\prime}(H) \quad \text{ by \eqref{eq:ColorTopgap} }  \\
&\le (1-\del) c^q_{0}(H^\prime) c^q_{k}(H) + \del \cdot c^q_{k}(H)  \quad \text{ by Claim~\ref{eqColorEpPrime} }  \\
&\le  c^q_{0}(H^\prime) c^q_{k}(H) \\
&\le c^q_k(H_{d,n})\,.\qedhere
\end{align*}

\end{proof}

\begin{proof}[\textbf{Small-2}: $n_1\le N_1$ and $k \ge \eps^\prime n$]
\hfill

Choose $0<\del' = \del'(\eps,  q,d) <1/2$ small enough that for $\lam \le 1-g(q,d,\eps)$, and any $d$-regular $G^\prime $ not containing a copy of $K_{d,d}$, we have $Z^q_{G'}(\lam) \le \frac{1-\del'}{1+2\del'} Z^q_{H_{d,|G'|}}(\lam)$, where $g(q,d,\eps)>0$ is the function from Lemma~\ref{lem:ColorlocalLim}. Such a $\del'$ exists by Proposition~\ref{prop:color-stability}.  
Choose $N= N(q, N_1, \eps, \eps')$ large enough that by Lemma~\ref{lem:ColorlocalLim}, for $n\ge N$, $\eps^\prime n< k <(1-\eps )nd/2q$, $n_1 \le N_1$ and $0\le r \le dn_1/2$ we have
\begin{equation}
\label{eqColorsmall2bound}
(1-\del') c^q_{k-r}(H_{d,n_2})   \le \lam^r \cdot  c^q_k(H_{d,n_2}) \le  (1+\del') c^q_{k-r}(H_{d,n_2}) 
\end{equation}
for some $\lam \le 1- g(q,d,\eps)$.  Now we bound
\begin{align*}
c^q_k(G) &= \sum_{r=0}^{dn_1/2}c^q_r(G')c^q_{k-r}(H)\\
&\leq (1+2\del')  \sum_{r=0}^{dn_1/2}c^q_r(G')\lam^r c^q_k(H)\quad \text{by \eqref{eqColorsmall2bound}} \\
&= (1+2\del') c^q_k(H)Z^q_{G'}(\lam) \\
&\le (1-\del') c^q_k(H)Z^q_{H'}(\lam) \\
&= (1 - \del') \sum_{r=0}^{dn_1/2}c^q_r(H')\lam^r c^q_k(H) \\
&\le \sum_{r=0}^{dn_1/2}c^q_r(H') c^q_{k-r}(H)\quad \text{by \eqref{eqColorsmall2bound}} \\
&= c^q_k(H_{d,n})\,.\qedhere
\end{align*}

\end{proof}

\begin{proof}[\textbf{Large}: $n_1> N_1$]
\hfill

This case is the same as that of {\bf Large} for matchings: we use stability in the form of Proposition~\ref{prop:color-stability} to lower bound $\frac{Z_{H_{d,n_1}}^q(\lam)}{ Z_{G'}^q(\lam) }$ and the local limit theorem, Theorem~\ref{thm:gdedenko}, to bound $ \frac{\pr_{\lam} [m(H_{d,n_2},\boldsymbol \chi) = \lceil kn_2/n \rceil]   }{    \pr_{\lam} [m(H_{d,n_2},\boldsymbol \chi) =k- s]   } $ and  $\pr_{\lam} \big[m(H',\boldsymbol \chi) = \lfloor kn_1/n \rfloor\big]$.
\end{proof}

\bibliography{stability}
\bibliographystyle{abbrv}

\end{document}